\documentclass[12pt]{amsart}
\usepackage[a4paper, margin={1in}]{geometry}
\usepackage{amssymb,amsmath}
\usepackage{enumitem}

\usepackage{tikz}
\usetikzlibrary{cd} 

\usepackage[english]{babel} 
\usepackage{csquotes}

\usepackage{xcolor} 
\usepackage[colorlinks=true, citecolor=blue, urlcolor=violet, breaklinks=true]{hyperref}

\newtheorem{theorem}{Theorem}[section]
\newtheorem{proposition}[theorem]{Proposition}
\newtheorem{lemma}[theorem]{Lemma}
\newtheorem{corollary}[theorem]{Corollary}

\theoremstyle{definition}
\newtheorem{example}[theorem]{Example}
\newtheorem{remark}[theorem]{Remark}

\newcommand{\defn}[1]{\textbf{#1}} 
\newcommand{\cvec}[1]{\begin{pmatrix}#1\end{pmatrix}}

\newcommand{\notimplies}{%
  \mathrel{\ooalign{\hidewidth$\not\phantom{=}$\hidewidth\cr$\implies$}}}
\newcommand{\notdiv}{%
  \mathrel{\ooalign{\hidewidth$\not\phantom{a}$\hidewidth\cr$\mid$}}}

\newcommand{\NN}{\mathbb{N}}
\newcommand{\ZZ}{\mathbb{Z}}
\newcommand{\RR}{\mathbb{R}}
\newcommand{\kk}{\Bbbk}

\newcommand{\pp}{\mathfrak{p}}

\newcommand{\set}[1]{\{#1\}}

\DeclareMathOperator{\Sym}{Sym}
\DeclareMathOperator{\Ann}{Ann}
\DeclareMathOperator{\Ext}{Ext}
\DeclareMathOperator{\Vol}{Vol}
\DeclareMathOperator{\cone}{cone}

\newcommand{\Rt}{\widetilde{R}}
\newcommand{\Rone}{\kk[R_1]}
\newcommand{\Pt}{\widetilde{P}}
\newcommand{\hs}{h^*}
\newcommand{\ths}{\widetilde{h}^*}
\newcommand{\MP}{M(P)}
\newcommand{\MPh}{\widehat{M}(P)}
\author{Lukas Katth\"an}
\address{Heinrich-Warnecke-Stra\ss{}e 6a, 37081 G\"ottingen, Germany}
\email{katthaen@math.uni-frankfurt.de}

\author{Kohji Yanagawa}
\address{Department of Mathematics, Kansai University, Suita, Osaka 564-8680, Japan}
\email{yanagawa@kansai-u.ac.jp}

\thanks{This work was partially supported by JSPS Grant-in-Aid for Scientific Research (C) 19K03456.}


\date{\today}

\title[Graded CM domains and lattice polytopes]{Graded Cohen-Macaulay domains and lattice polytopes with short $h$-vector}

\begin{document} 
\begin{abstract}
	Let $P$ be a lattice polytope with $h^*$-vector $(1, h^*_1, h^*_2)$.
	In this note we show that if $h_2^* \leq h_1^*$, then $P$ is IDP.  
	More generally, we show the corresponding statements for semi-standard graded Cohen-Macaulay domains over algebraically closed fields.
\end{abstract}

\maketitle

\section{Introduction}
Let  $R=\bigoplus_{i \in \NN} R_i$ be a noetherian graded commutative ring. 
Throughout the paper, we assume that $\kk:= R_0$ is an algebraically closed field. 
If $R =\kk[R_1]$, that is, $R$ is generated by $R_1$ as a $\kk$-algebra, we say $R$ is \defn{standard graded}.
If $R$ is finitely generated as a $\kk[R_1]$-module, we say $R$ is \defn{semi-standard graded}.

If $R$ is a semi-standard graded ring of Krull dimension $d$, its Hilbert series is of the form  
\[
	\sum_{i \in \NN} (\dim_\kk R_i) t^i =\frac{h_0+h_1 t + \cdots +h_st^s}{(1-t)^d}
\]
for some integers $h_0, h_1, \ldots, h_s$ with $\sum_{i=0}^s h_i \ne 0$ and $h_s \ne 0$.  
We call the vector $(h_0,h_1,\ldots, h_s)$ the \defn{$h$-vector} of $R$. 
We always have $h_0=1$ and $\deg R=\sum_{i=0}^s h_i$. If $R$ is Cohen--Macaulay, then $h_i \ge 0$ for all $i$. 
The following is our main result:

\begin{theorem}\label{thm:main}
	Let $R$ be a semi-standard graded Cohen-Macaulay domain (with $R_0 = \kk=\overline{\kk}$) and $h$-vector $(h_0, h_1, h_2)$.  If $h_2 \leq h_1$, then $R$ is standard graded.
\end{theorem}

In fact, we will prove a more general version, see Theorem~\ref{thm:main_mg} below.

An important class of semi-standard graded Cohen-Macaulay domains are the Ehrhart rings of lattice polytopes, which we now recall.
Let $P \subset \RR^d$ be a lattice polytope.
Its \defn{Ehrhart ring} $\kk[P]$ is the monoid algebra of the monoid of lattice points in the cone $C = \cone(\set{1} \times P ) \subset \RR^{d+1}$  over $P$.
The additional coordinate in the construction of $C$ yields a natural grading on $\kk[P]$, such that $\kk[P]$ is semi-standard graded, and 
its Hilbert series is the Ehrhart series of $P$.
In particular, the $h$-vector of $\kk[P]$ is the $h^*$-vector of $P$. Hence the Krull dimension $\dim \kk[P]$ equals $\dim P+1$. 

It is well-known that $\kk[P]$ is a normal domain,
and by Hochster's Theorem \cite[Theorem 1]{hochster}, it is Cohen-Macaulay.
We refer to the reader to the monograph by Bruns and Gubeladze \cite{BG} for more information on Ehrhart rings.
The index of the last non-zero entry of the $h^*$-vector is called the \defn{degree} of $P$.  We always have $\deg (P) \leq \dim (P)$. The $h^*$-vector of $P$ is sometimes denoted by $(h_0^*, h_1^*, \ldots, h^*_{\dim(P)})$, even if $\deg(P) < \dim (P)$. In this case, $h_i^*=0$ for all $i > \deg(P)$. 
We also remark that there is no direct relation between $\deg(P)$ and $\deg(\kk[P]) =\sum_{i=0}^d h_i^*$, the latter being the multiplicity of $\kk[P]$, which also equals the normalized volume $\Vol(P)$ of $P$.

A lattice polytope $P$ is called \defn{IDP} (an abbreviation for ``integer decomposition property'') if for every $k \in \NN$ and every lattice point $p \in kP \cap \ZZ^d$, there exist $k$ lattice points $p_1, \dotsc, p_k \in P \cap \ZZ^d$ with $p = \sum_i p_i$.
Clearly, $P$ is IDP if and only if $\kk[P]$ is standard graded.
Hence we obtain the following combinatorial version of our main result (here we do not have to assume that $\kk$ is algebraically closed, since we can replace $\kk[P]$ by $\overline{\kk}[P] \cong \overline{\kk} \otimes_\kk \kk[P]$):
\begin{corollary}\label{cor:main}
	Let $P \subset \RR^d$ be a lattice polytope of degree $2$ with $h^*$-vector $(1, h^*_1, h^*_2)$.
If $h^*_2 \leq h^*_1$, then $P$ is IDP.
\end{corollary}

Note that if $P \subset \RR^2$ is a lattice polygon, then it has degree at most 2, and it always satisfies $h^*_2 \leq h^*_1$.
Therefore, this corollary can be seen as an extension of the well-known fact that lattice polygons are IDP. See also Remark~\ref{Rem 3.3} (1) below. 

We give an example to show that the bound in Corollary~\ref{cor:main} is sharp:

\begin{example}\label{eq:reeves}
	Let $P$ be the $3$-simplex with vertices 
	\[
	\cvec{0\\0\\0},
	\cvec{1\\0\\0},
	\cvec{0\\1\\0} \text{ and }
	\cvec{1\\1\\2}.
	\]
	It is a Reeves-simplex (cf.~\cite[Example 2.56(a)]{BG}) and its $h^*$-vector is $(1,0,1)$.
	It is not IDP, and hence the bound $\hs_2 \leq \hs_1$ is sharp.
\end{example}


\subsection*{Acknowledgements}
The authors thank Kazuma Shimomoto for many inspiring discussion throughout this project.

\section{Proofs of the main results}
As before, let  $R=\bigoplus_{i \in \NN} R_i$ be a noetherian graded commutative ring such that $\kk:= R_0$ is an algebraically closed field. 
We are going to regard $R$ as a module over $S := \Sym_\kk R_1$.
Note that $S$ is isomorphic to the polynomial ring $\kk[x_1, \ldots, x_n]$ with $n=\dim_\kk R_1$.
Moreover, $R$ is standard graded if and only if $R$ is a quotient ring of $S$, and $R$ is semi-standard graded if and only if $R$ is finitely generated as an $S$-module. 

For a finitely generated graded $S$-module $M$ and natural numbers $i,j \in \NN$, set 
$$\beta_{i,j}^S(M):= \dim_\kk[\operatorname{Tor}_i^S(\kk, M)]_j.$$ 
In particular, $\beta_{0,j}^S(M)$ is the number of $S$-module generators for $M$ in degree $j$.  

\subsection{A bound on the degrees of the generators}
Assume that $R$ is semi-standard graded, and has the $h$-vector $(h_0, h_1, \ldots, h_s)$. 
The goal of this section is to obtain a bound on the degrees of the generators of $R$ as an $S$-module.
If $R$ is Cohen--Macaulay, it is well-known that the generators have degree at most $s$. 

Our result is a sufficient criterion when this bound can be improved by one:
\begin{theorem}\label{thm:main_mg}
	Let $R$ be a semi-standard graded Cohen-Macaulay domain, $S := \Sym_\kk R_1$ and with $h$-vector $(1 = h_0, h_1, h_2, \dotsc, h_s)$.
	Then it holds that 
	\[ \beta^S_{p,p+s}(R) = 0 \text{ for } 0 \leq p \leq h_1 - h_s.\]
	In particular, if $h_s \leq h_1$, then $R$ is generated by elements of degree $\leq s-1$ as an $S$-module.
\end{theorem}
Note that Theorem~\ref{thm:main} amounts to the special case $s = 2$ and $p = 0$.
This result and its proof have been inspired by Green's Theorem of the Top Row, \cite[Theorem 4.a.4]{G1}.

For the proof of Theorem~\ref{thm:main_mg}, we are going to use the following version of Green's vanishing theorem:
\begin{theorem}[{\cite[Theorem 1.1]{EKsyz}}]\label{thm:vanish}
	Let $\pp \subset S$ be a homogeneous prime ideal, which does not contain any linear forms.
	Let $M$ be a torsion free finitely generated graded $S/\pp$-module and let $q \in \ZZ$ be the minimal integer such that $M_q \neq 0$.
	Then it holds that
	\[ \beta^S_{p,p+q}(M) = 0 \quad\text{ for }\quad p \geq \dim_\kk M_q. \]
\end{theorem}
\noindent In addition, we need the following result:
\begin{theorem}[{\cite[Theorem 6.18]{BG}, \cite[Theorem 4.4.5]{BH}}]\label{prop:dual}
	Let $M$ be a finitely generated graded Cohen-Macaulay module over $S=\kk[x_1, \ldots, x_n]$ with $d = \dim M$.
	Define $M' := \Ext_S^{n - d}(M, \omega_S)$.
	Then
	\begin{enumerate}
		\item $M'$ is also Cohen-Macaulay, $\Ann M' = \Ann M$ 
and $M'' \cong M$.
		\item $\beta^S_{p,q}(M') = \beta^S_{n - d - p, n - q}(M)$.
		\item $H_{M'}(t) = (-1)^{d}H_M(t^{-1})$.  
	\end{enumerate}
\end{theorem}
\noindent Here, $\omega_S := S(-n)$ denotes the canonical module of $S$, and $H_M(t)$ denotes the Hilbert series $\sum_{i\in \ZZ} (\dim_\kk M_i) \cdot t^i$ of $M$.

\begin{proof}[Proof of Theorem~\ref{thm:main_mg}]
	Let  $d := \dim R$.
	By \cite[Proposition 3.6.12]{BH}, $\omega_R := \Ext_S^{n-d}(R, \omega_S)$ is a canonical module for $R$.
	Further, note that $\omega_R = R'$ in the notation of Theorem~\ref{prop:dual}.
	By that theorem,
	it holds that
	\[ \beta^S_{p,p+q}(R) = \beta^S_{n-d-p, \, n-(p+q)}(\omega_R) = \beta^S_{n-d-p, \, (n-d-p)+ (d-q)}(\omega_R).\]
	and the Hilbert series of $\omega_R$ is
	\[
		\frac{ h_s  t^{d-s}+  h_{s-1} t^{d-s+1}+ \dotsb +  h_0 t^d}{(1-t)^d}.
	\]
	In particular, $\omega_R$ has no elements in degrees below $d-s$ and we have that $\dim_\kk (\omega_R)_{d-s} = h_s$.
	Now, since $R$ is a domain and $\omega_R$ is a canonical module, it is torsion free 
over $R$ by Theorem~\ref{prop:dual} (1), and thus it satisfies the hypothesis of Theorem~\ref{thm:vanish}.
	Applying that result to $\omega_R$ yields
	that 
	\[
		\beta^S_{p,p+s}(R) = \beta^S_{n-d-p, \, (n-d-p) + (d-s)}(\omega_R) = 0  \quad\text{ if } (n-d) - p \geq h_s
	\]
	Finally, note that $h_1 = n - d$, and the proof is complete.
\end{proof}

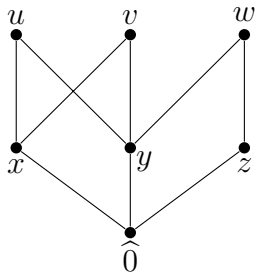
\begin{figure}[ht]
	\begin{tikzpicture}[every node/.style={circle,inner sep=0pt, fill=black,  minimum size=1.5mm}, scale = 1.5]
		\path (0,0)
			node[label=south:$\widehat{0}$] (0) at (0,0.25) {}
			node[label=south:$x$] (x) at (-1,1) {}
			node[label=south east:$y$] (y) at (0,1) {}
			node[label=south:$z$] (z) at (1,1) {}
			node[label=north:$u$] (u) at (-1,2) {}
			node[label=north:$v$] (v) at (0,2) {}
			node[label=north:$w$] (w) at (1,2) {};
	
		\draw (0) -- (z) -- (w) -- (y) -- (u) -- (x) -- (v) -- (y) -- (0) -- (x);
	\end{tikzpicture}
	\caption{The poset $P$ of Remark~\ref{rem:poset}}\label{fig:poset}
\end{figure}

\begin{remark}\label{rem:poset}
Another important example of a semi-standard graded ring appearing in combinatorial commutative algebra  is the \defn{face ring} $A_P$ of a simplicial poset $P$.
	See \cite{Stan2} for details.
	For the simplicial poset $P$ given in Figure \ref{fig:poset}, we have 
	\[A_P \cong \frac{\kk[x,y,z,u,v]}{(xz, uz, vz, uv, xy-u-v)},\]
	where $\deg x =\deg y=\deg z=1$ and $\deg u=\deg v=2$, and $\widehat{0}, w \in P$ correspond to $1, yz \in A_P$, respectively. It is easy to see that $A_P$ is a 2-dimensional  Cohen-Macaulay reduced semi-standard graded ring with the $h$-vector $(1,1,1)$, but it is \defn{not} standard graded.
	It means that Theorem~\ref{thm:main_mg} indeed requires  the assumption that $R$ is a domain.
\end{remark}

\section{Further Discussion on Ehrhart rings}

\subsection{Direct Applications of Theorem \ref{thm:main_mg}}
We now apply Theorem~\ref{thm:main_mg} in the setting of Ehrhart theory.

Let $P \subset \RR^n$ be a lattice polytope.
We write $\MP \subset \ZZ^{n+1}$ for the affine monoid generated by the lattice points in $P \times \set{1} \subset \RR^{n+1}$,
and $\MPh\subset \ZZ^{n+1}$ for its integral closure inside $\ZZ^{n+1}$.
Let $R=\kk[P]$ be the Ehrhart ring of $P$, and $\kk[R_1]$ its subalgebra generated by $R_1$. Then $R$ and $\kk[R_1]$ are the monoid algebras of the monoids $\MPh$ and $\MP$, respectively.   It is well-known that $\MPh$ is generated by elements of degree at most $\min(\deg(P), \dim(P)-1)$ as a module over $\MP$ (cf. \cite[Theorem 2.52]{BG}).
Equivalently, $R=\kk[P]$ is generated by elements at most that degree
as $\Rone$-module, and hence in particular as a $\kk$-algebra. 

Since $\kk[P]$ is always Cohen-Macaulay, Theorem~\ref{thm:main_mg} allows us to improve this bound under an additional assumption as follows. Clearly, this generalizes Corollary~\ref{cor:main}. 
\begin{corollary}\label{thm:ehrhart}
	Let $P \subset \RR^n$ be a lattice polytope of degree $s$ with $h^*$-vector $(1, h^*_1, \dotsc, h^*_s)$.
	If $h^*_s \leq h^*_1$, then $\MPh$ is generated by elements of degrees $\leq s - 1$ as an $\MP$-module.
\end{corollary}
Let $P^\circ$ be the relative interior of $P$. The lattice points in $P^\circ$ are closely related to the canonical module of $\kk[P]$ (c.f. \cite[Theorem~6.3.5 (b)]{BH}). 
In general, it holds that $\hs_{\dim P} \leq \hs_1$ (because $\hs_{\dim P} = \#(P^\circ \cap \ZZ^n) \leq \#(P \cap \ZZ^n) - (\dim P +1) = \hs_1$),  therefore this corollary extends the bound mentioned above.

If $P$ is IDP, then $R = \kk[P]$ is the quotient ring of $S = \Sym_\kk R_1$ by a certain prime ideal $I \subseteq S$, which is called the \defn{toric ideal} of $P$.
It is known that $I$ is generated by polynomials of degree at most $\deg(P) + 1\leq \dim(P) + 1$ (Sturmfels, cf. \cite[Corollary 7.27]{BG}), and again we can improve these bounds by one:
\begin{corollary}
	Let $P$ be an IDP lattice polytope and let $I \subset S$ be its toric ideal.
	\begin{enumerate}
	\item If $h^*_s \leq h^*_1 - 1$, then $I$ is generated in degrees $\leq \deg(P)$.
	\item If $P$ is not a clean simplex, then $I$ is generated in degrees $\leq \dim(P)$.
	\end{enumerate}
\end{corollary}\label{Cor 3.2}
Recall that a \defn{clean simplex} is a lattice simplex where the only lattice points on its boundary are the vertices.
\begin{proof}
\begin{enumerate}
	\item Apply Theorem~\ref{thm:main_mg} to $R = \kk[P]$ with $p = 1$.
	\item If $I$ has a generator in degree $\dim(P) + 1$, then by the result mentioned above it holds that $\deg(P) = \dim(P)$.
	Moreover, the hypothesis of part (1) needs to be violated, hence it holds that $\hs_{\dim(P)} \geq \hs_1$.
	It follows that $\hs_{\dim(P)} = \hs_1$, which is equivalent to $P$ being a clean simplex.\qedhere
\end{enumerate}
\end{proof}

\begin{remark}\label{Rem 3.3}
(1)  Let $P \subset \RR^2$ be a lattice polygon. Then we have $\deg P \le 2$ and $P$ is IDP. Moreover, Koelman \cite{koelman} showed that the toric ideal of $P$ is generated by quadrics if and only if $h^*_2 < h^*_1$. 
Hence Corollary~\ref{Cor 3.2} is an extension of one implication of his result.
In particular, the result of \cite{koelman} shows that the bound $h^*_2 < h^*_1$ is sharp.

(2)  In  \cite{Schenck}, H. Schenck also applied the theory of M. Green to the study of Ehrhart rings $\kk[P]$. However, the focus of \cite{Schenck} is different from ours. More precisely, he always assumed that $\kk[P]$ is standard graded (i.e., $P$ is IDP), and treated the case the toric ideal is generated by quadrics.    

(3)  By an argument similar to the above, in the situation of Theorem~\ref{thm:main}, if $h_2 < h_1$, then $R$ is standard graded, and its ``defining ideal'' is generated by quadrics. 
If further $\operatorname{char}(\kk)=0$, $R$ is \defn{Koszul} (actually, \defn{absolutely Koszul}) by \cite[Theorem~5.2 (1)]{CINR}. In other words, we can slightly weaken the assumption of 
 \cite[Theorem~5.2 (1)]{CINR} to that ``$R$ is semi-standard graded''.  
So if $\operatorname{char}(\kk)=0$ and a lattice polytope $P$ has the $h^*$-vector 
$(h_0^*, h_1^*, h_2^*)$ with $h_2^* < h_1^*$, then $\kk[P]$ is (absolutely) Koszul.

\end{remark}

\subsection{Combinatorial proofs}
Corollary~\ref{cor:main} is a purely combinatorial statement, and hence one might hope for a combinatorial proof.
As a first step, we prove a weak variant of Corollary~\ref{cor:main} which admits an elementary proof.

We remind the reader that a lattice polytope $P \subset \RR^n$ is called \defn{spanning} \cite{HKN1}, if $(P \times \set{1}) \cap \ZZ^{n+1}$ generates the lattice $\ZZ^{n+1}$.
Every IDP polytope is spanning, but the converse is far from being true. 
Algebraically, for the Ehrhart ring $R=\kk[P]$, $P$ is spanning if and only if the field of fractions of $R$ coincides with that of  $\Rone$.

\begin{proposition}\label{prop:spanning}
	Let $P \subset \RR^n$ be a $d$-dimensional lattice polytope with $\hs$-vector $(\hs_0, \hs_1, \dotsc)$.
	If $\hs_1 + \hs_d \geq \sum_{i=2}^{d-1} \hs_i$, then $P$ is spanning.
\end{proposition}
\begin{proof}
	We show the contrapositive.
	Assume that $P$ is not spanning, and let $q >1$ be the index of the lattice generated by the lattice points in $P$.
	Further, let $\Pt$ be the polytope $P$ considered in the lattice generated by its lattice points (see \cite{HKN1}).
	We write $\ths$ for the $\hs$-vector of $\Pt$.
	It holds that
	\begin{equation}\label{eq:sp}
	\sum_{i=0}^d \hs_i = \Vol(P) = q \Vol(\Pt) = q \sum_{i=0}^d \ths_i.
	\end{equation}
	Moreover, it holds that $\hs_1=\ths_1, \hs_d=\ths_d$ and $\hs_i \geq \ths_i$ for $1 \leq i \leq d$ (see Section 3.2 of \cite{HKN1}).
	Now \eqref{eq:sp} implies that
	\[ 0 \leq q \sum_{i=2}^{d-1} \ths_i =  \sum_{i=2}^{d-1} \hs_i - (q-1)(1+\hs_1+\hs_d) \leq  \sum_{i=2}^{d-1} \hs_i - (1+\hs_1+\hs_d)\]
	and thus $\hs_1 + \hs_d < \sum_{i=2}^{d-1} \hs_i$.
\end{proof}

In the next corollary, (1) is just a weak version of Corollary~\ref{cor:main}, but (2) and (3) are new.

\begin{corollary} With the above notation, the following hold. 
	\begin{enumerate}
		\item If $\deg P = 2$ and $\hs_1 \geq \hs_2$, then $P$ is spanning.
		\item If $\dim(P) = 3$ and $\hs_1 + \hs_3 \geq \hs_2$, then $P$ is spanning.
		\item If $\dim(P) = 4$, $\deg(P) \geq 3$, and $\hs_1 + \hs_4 \geq \hs_2 + \hs_3$, then $P$ is spanning. In this case, it holds that $\hs_1 = \hs_2 = \hs_3 = \hs_4$.
	\end{enumerate}
\end{corollary}
\begin{proof}
	Only the very last statement is not immediate from Proposition~\ref{prop:spanning}.
	If $\dim(P) = 4$, then $\hs_4 \leq \hs_1$.
	By assumption and Proposition~\ref{prop:spanning}, $P$ is spanning, and hence by
	\cite[Theorem 1.4]{UPP} it holds that $\hs_1 \leq \hs_i$ for $1 \leq i < \deg(P)$.
	As $\deg(P) \geq 3$ it holds that $\hs_1 \leq \hs_2$ and $\hs_3 > 0$, and thus $\hs_1 < \hs_2 + \hs_3$.
	It follows that $\hs_4 \neq 0$, so $\deg(P) = 4$.
	Hence we have that $\hs_4 \leq \hs_1 \leq \hs_2$ and $\hs_1 \leq \hs_3$. These inequalities, together with $\hs_1 + \hs_4 \geq \hs_2 + \hs_3$, imply that $\hs_1 = \hs_2 = \hs_3 = \hs_4$.
\end{proof}

Unfortunately, these are the only cases where Proposition~\ref{prop:spanning} can be applied, due to the following observation:
\begin{proposition}
	If $d:=\dim P \geq 5$ and $\deg P \geq 3$, then $\hs_1 + \hs_d < \sum_{i=2}^{d-1} \hs_i$.
\end{proposition}
\begin{proof}
	Assume the contrary that the claimed inequality does not hold. Then $P$ is spanning by Proposition~\ref{prop:spanning}.
	Now, by \cite[Theorem 1.4]{UPP} it holds that $\hs_1 \leq \hs_i$ for $1 \leq i < \deg(P)$.
	
	We distinguish two cases.
	First, assume that $\deg(P) \leq 4$.
	Then it holds that $\hs_d = 0$, $\hs_1 \leq \hs_2$ and $\hs_3 > 0$ (since $\deg(P) \geq 3$), thus we have that
	\[ 
	\hs_1 + \hs_d = \hs_1 < \hs_2 + \hs_3 \leq \sum_{i=2}^{d-1} \hs_i.
	\]
	Next, assume that $\deg(P) \geq 5$.
	Then $\hs_d \leq \hs_1$ and, as before, $\hs_1 \leq \hs_2$, $\hs_1 \leq \hs_3$ and $\hs_4 > 0$.
	Therefore
	\[
	\hs_1 + \hs_d \leq \hs_2 + \hs_3 < \hs_2 + \hs_3 + \hs_4 \leq \sum_{i=2}^{d-1} \hs_i.
	\]
	In both cases, we obtain a contradiction.
\end{proof}

\subsection{About Polytopes of degree \texorpdfstring{$2$}{2}}
One can combine our Corollary~\ref{cor:main} with the results of \cite{HY18} to obtain the following web of implications for lattice polytopes of degree $2$:
\begin{theorem}
	Let $P \subset \RR^n$ be a lattice polytope of degree 2 with $\hs$-vector $(1, \hs_1, \hs_2)$, and let  $\Pt$ denote the polytope $P$ considered as a lattice polytope inside the lattice generated by the lattice points in $P$.
Then the following implications hold:
	\[
	\begin{tikzcd}[column sep=small, arrows={Rightarrow}]
			\hs_1 \geq \hs_2 \ar[rd] \ar[rr]
			&& \hs_1+1 \notdiv \hs_2 \ar[rr] &&
			\deg\Pt \neq 1 \ar[rr,Leftrightarrow] && \text{level}\\
			& \text{IDP} \ar[rr]	&&	\text{spanning} \ar[ru]	&&&
	\end{tikzcd}
	\]
\end{theorem}
Here, we say that a lattice polytope $P$ is \defn{level} if its Ehrhart ring $R=\kk[P]$ is level, that is, its canonical module $\omega_R$ is generated in a single degree as an $R$-module. The levelness of $P$ is a combinatorial property of the monoid $\MPh$ (c.f. \cite[Proposition 4.3]{HY18}), and does not depend on the base field $\kk$.  

\begin{proof}
\begin{description}[font=\normalfont\enquote]
	\item[$\hs_1 \geq \hs_2\!\implies\!\hs_1+1 \notdiv \hs_2$] This is elementary.
	
	\item[$1 + h^*_1 \notdiv h^*_2\!\implies\!\deg(\Pt) \neq 1$]
		We show the contrapositive.
		Assume that $\deg(\Pt) = 1$.
		Denote the $h^*$-vector of $\Pt$ by $\tilde{h}^*$.
		The volume of $\Pt$ divides the volume of $P$, since the latter is normalized with respect to a finer lattice.
		Thus we have that 
		\[ (1 + \tilde{h}^*_1 + \tilde{h}^*_2) \mid (1 + h^*_1 + h^*_2) \]
		On the other hand, we have that $\tilde{h}^*_1 = h^*_1$ and by assumption, $\tilde{h}^*_2 = 0$.
		It follows that $(1+h^*_1) \mid h^*_2$.

	\item[$\hs_1 \geq \hs_2\!\implies$IDP] This is Corollary~\ref{cor:main}.
	\item[IDP$\implies$spanning] This is well-known (and elementary), and its does not need the assumption $\deg P \leq 2$.
	\item[spanning$\implies\! \deg(\Pt) \neq 1$]
		$\deg \Pt= \deg P = 2 \neq 1$.
	\item[$\deg(\Pt) \neq 1\!\implies$level]
		Under this hypothesis, the degree of $\Pt$ is either $0$ or $2$. We distinguish those cases:
		\begin{description}[font=\normalfont]
			\item[$\deg(\Pt) = 0$] In this case $h^*_1 = 0$, so the claim follows form \cite[Lemma 2.1]{HY18}.
			\item[$\deg(\Pt) = 2$] Let $R$ and $\Rt$ denote the Ehrhart rings of $P$ and $\Pt$.
			Then $\deg\Rt = 1 + h^*_1 + \tilde{h}^*_2 > 1 + h^*_1$ by assumption, and thus $R$ is level by \cite[Proposition 3.4]{HY18}.
		\end{description}
		\item[level$\implies\! \deg(\Pt) \neq 1$] This follows from the following more general Lemma~\ref{lem:level} below.\qedhere
\end{description}
\end{proof}

\begin{lemma}\label{lem:level}
	Let $P \subset \RR^n$ be a lattice polytope. 
	If $\kk[P]$ is level, then $\deg(\Pt) \neq \deg P - 1$.
\end{lemma}
\begin{proof}
	Let $c(P) := \min\set{\ell \in \ZZ_{>0} \colon \ell P^\circ \cap \ZZ^n \neq \emptyset}$ (sometimes this is called the \defn{codegree} of $P$).
	It is well-known that $\deg(P) = \dim(P) + 1 - c(P)$.
	
	We are going to use \cite[Proposition 4.3]{HY18}, which we recall for convenience: If $P$ is level, then for any $k \geq c(P)$ and $\alpha \in kP^\circ \cap \ZZ^n$, there exist a $\beta \in c(P)P^\circ \cap \ZZ^n$ and $\gamma \in (k - c(P))P \cap \ZZ^n$ such that
	\[ \alpha = \beta +  \gamma. \]
	
	Now, assume that $\deg \Pt = \deg P - 1$, and note that this implies $c(\Pt) = c(P) + 1$. 
	Let $L \subset \ZZ^n$ be the sublattice spanned by 
the lattice points in $P$. 
	As $P \neq \Pt$, this is a proper sublattice of $\ZZ^n$.
	Choose $\alpha \in c(\Pt)P^\circ \cap L$.
	Then, if $P$ were level, there would exist $\beta$ and $\gamma$ as above.
	As $\beta \in c(P)P^\circ \cap \ZZ^n$, it follows that $\beta \notin L$ (because $c(P)\Pt$ has no interior lattice points).
	Further, $\gamma$ lies in $(c(\Pt) - c(P))P = P$ and thus $\gamma \in L$.
	But this contradicts $\beta + \gamma = \alpha \in L$.
\end{proof}

We provide some examples to show that all the implications are strict and that there are no other implications. In each example, the claimed properties can  conveniently be verified using \texttt{normaliz} \cite{Normaliz}.

\begin{example}[$\hs_1+1\notdiv\hs_2 \notimplies$ spanning, IDP, $\hs_1 \geq \hs_2$]\label{ex:sp}
	Consider the $4$-polytope $P$ with vertices
	\[
	\cvec{0\\0\\0\\0},
	\cvec{1\\1\\0\\0},
	\cvec{1\\0\\1\\0},
	\cvec{1\\0\\0\\1},
	\cvec{0\\1\\1\\0},
	\cvec{0\\1\\0\\1} \text{ and }
	\cvec{0\\0\\1\\1}.
	\]
	Its $\hs$-vector is $(1,2,5)$, so it satisfies $\hs_1+1\notdiv\hs_2$, but $\hs_1 \ngeq \hs_2$.
	To see that it is not spanning (and thus not IDP), consider the vector
	\[ v := \cvec{1\\1\\1\\0} = \frac{1}{2}\left( 
	\cvec{0\\0\\0\\0} + 
	\cvec{1\\1\\0\\0} + 
	\cvec{1\\0\\1\\0} + 
	\cvec{0\\1\\1\\0} \right)
	\]
	It lies in $2P \cap \ZZ^4$, 
	but the sum of its coordinates is odd, while the coordinate sum of each vertex of $P$ is even.
	Hence $v$ cannot lie in the lattice spanned by them.
\end{example}

\begin{example}[IDP and spanning $\notimplies \hs_1+1\notdiv\hs_2, \hs_1 \geq \hs_2$]\label{eq:idp}
	Let $P$ be the $3$-simplex with vertices 
	\[
	\cvec{0\\0\\0},
	\cvec{1\\0\\0},
	\cvec{0\\4\\0} \text{ and }
	\cvec{1\\0\\3}.
	\]
	It is IDP, and its $\hs$-vector is $(1,5,6)$, so $\hs_1+1\mid\hs_2$ and $\hs_1 \ngeq \hs_2$
\end{example}

\begin{example}[spanning $\notimplies$ IDP]\label{eq:va}
	It is well-known that this implication does not hold in general.
	For an example with degree $2$, see \cite[Exercise~2.24]{BG}.
	This is a very-ample (and thus spanning) $3$-polytope which is not IDP.
	Its $\hs$-vector is $(1,4,5)$, so it has degree $2$.
\end{example}

\begin{example}[$\deg\Pt \neq 1 \notimplies \hs_1+1\notdiv \hs_2$, spanning]
	Let $P$ be the polytope of Example~\ref{eq:reeves} with $\hs$-vector $(1,0,1)$.
	In this case $\Pt$ is a unit simplex and thus $\deg \Pt = 0 \neq 1$.
	However, $P$ is not spanning and it holds that $\hs_1+1\mid \hs_2$.
\end{example}

\begin{example}[$\hs_1+1\notdiv \hs_2$ and IDP $\notimplies \hs_1 \geq \hs_2$]
	Let $P$ be the $3$-simplex with vertices 
	\[
	\cvec{0\\0\\0},
	\cvec{1\\0\\0},
	\cvec{0\\4\\0} \text{ and }
	\cvec{1\\0\\4}.
	\]
	Its $\hs$-vector is $(1,6,9)$, so it satisfies $\hs_1+1\notdiv \hs_2$, but $\hs_1 \ngeq \hs_2$.
	Moreover, it is IDP.
\end{example}


\end{document}